\documentclass{article}
\usepackage{amsmath,amsthm,amsfonts,graphicx}
\usepackage{multirow}
\usepackage{slashbox}
\usepackage{multirow}
\def\smallddots{\mathinner{\raise7pt\hbox{.}\raise4pt\hbox{.}\raise1pt\hbox{.}}}
\def\smallsdots{\mathinner{\raise1pt\hbox{.}\raise4pt\hbox{.}\raise7pt\hbox{.}}}

\DeclareMathOperator{\diag}{diag}

\DeclareMathOperator{\rank}{rank}

\numberwithin{equation}{section}
\numberwithin{table}{section}
\newtheorem{theorem}{Theorem}[section]
\newtheorem{lemma}{Lemma}[section]

\newtheorem{corollary}{Corollary}[section]

\newtheorem{definition}{Definition}[section]

\newtheorem{remark}{Remark}[section]

\begin{document}

\title{\bf Estimating the Norms of
Random Circulant and Toeplitz Matrices
and Their Inverses
\thanks {Some results of this paper have been presented at the 
ACM-SIGSAM International 
Symposium on Symbolic and Algebraic Computation (ISSAC '2011), San Jose, CA, 2011,
the 7th International Congress on Industrial and Applied Mathematics 
(ICIAM 2011), in Vancouver, British Columbia, Canada, July 18-22, 2011,
the SIAM International Conference on Linear Algebra,
in Valencia, Spain, June 18-22, 2012, and the 
18th Conference of the International Linear Algebra
Society (ILAS'2013), Providence, RI, 2013
}}

\author{Victor Y. Pan$^{[1, 2],[a]}$ and Guoliang Qian$^{[2],[b]}$
\and\\
$^{[1]}$ Department of Mathematics and Computer Science \\
Lehman College of the City University of New York \\
Bronx, NY 10468 USA \\
$^{[2]}$ Ph.D. Programs in Mathematics  and Computer Science \\
The Graduate Center of the City University of New York \\
New York, NY 10036 USA \\
$^{[a]}$ victor.pan@lehman.cuny.edu \\
http://comet.lehman.cuny.edu/vpan/  \\
$^{[b]}$ gqian@gc.cuny.edu \\
} 
 \date{}

\maketitle


\begin{abstract}

We estimate the norms of standard Gaussian random Toeplitz and 
circulant matrices and their inverses, mostly by means of combining
some basic techniques of linear algebra. In the case of circulant 
matrices we obtain sharp probabilistic estimates, which show that 
these matrices are expected to be very well conditioned. Our 
probabilistic estimates for the norms of standard Gaussian random 
Toeplitz matrices are within a factor of $\sqrt 2$
from those in the circulant case. We also achieve partial progress 
in estimating the norms of Toeplitz inverses. Namely we yield 
reasonable probabilistic upper estimates for these norms assuming 
certain bounds on the absolute values of two corner entries of the 
inverse. Empirically we observe that the condition numbers of random Toeplitz and 
general matrices tend to be of the same order. As the matrix size grows, 
these numbers grow equally slowly, although faster than for random circulant matrices.

\end{abstract}





\section{Introduction}\label{sintro}


Estimating the condition numbers $\kappa(A)=||A||~||A^{-1}||$
of random structured matrices $A$
is a well known challenge  \cite{SST06},
linked to the design  of efficient randomized matrix algorithms, e.g.,
in the papers \cite{HMT11}, \cite{XXG12}, \cite{PQZ13}. 
We seek such estimates for standard Gaussian random Toeplitz and circulant
matrices, mostly by employing and combining some basic techniques of linear algebra.
In the case of circulant matrices $A$
our sharp probabilistic estimates 
for the norms $||A||$ and $||A^{-1}||$
show that the matrices are expected to be very well conditioned.
In the case of  Toeplitz matrices $A$ 
our estimates for the norm $||A||$
 are  within a factor of $\sqrt 2$
from the bounds in the circulant case.
Estimating the norms $||A^{-1}||$ turns out to be harder, however.
 Namely we obtain reasonable probabilistic upper bounds on that norm,
and consequently on the condition number $\kappa(A)=||A||~||A^{-1}||$,
 assuming that the norms of the first row and the first column
 of the inverse do not exceed dramatically 
the absolute values of their
two corner entries, $(A^{-1})_{1,n}$ and  $(A^{-1})_{n,1}$, 
respectively.

For some large and important special classes
of  $n\times n$
Toeplitz matrices   
the condition numbers
 grow exponentially in $n$ as $n\rightarrow \infty$
 \cite{BG05}, but 
in our tests with both
random general and random
Toeplitz matrices 
 their condition numbers 
consistently grew with the same reasonably slow rate as $n$ grew large,
although the growth was still
significantly faster than 
in the case of circulant matrices.

Our study can be immediately extended to the cases of factor circulant 
and Hankel matrices and partly to the case of other 
than Gaussian probability distributions.


We organize our paper as follows.
We recall some definitions and basic results on general matrix computations
in the next 
section and on Toeplitz and circulant matrices in Section \ref{stplc}.
In Section \ref{srnbtc} we deduce sharp estimates 
on the norms of circulant and Toeplitz matrices
in terms of their generating vectors.
In Section \ref{sircrc} we similarly estimate
the norms of the inverses of these matrices.
We recall the definition of Gaussian random matrices and some basic
facts about them in Section \ref{grm}.
In Section \ref{sgrtc} we extend our 
estimates of  Sections \ref{srnbtc} and \ref{sircrc}
to standard Gaussian random 
 Toeplitz and  circulant  matrices and their inverses.
 In Section \ref{sexp} we cover numerical tests, which are
the contribution of the second author.
In the Appendix we discuss nondegeneration of random matrices.




\section{Some definitions and basic results}\label{sdef}


In ths section we recall some customary definitions 
and basic properties on matrix computations
\cite{GL96}, \cite{S98}.
$A^T$ is the transpose  of a 
matrix $A$, and $A^H$ is its Hermitian  transpose.
$A^H=A^T$ for a real matrix $A$. $A^T=A$ if $A$ is a real symmetric matrix,
 $A^H=A$ if $A$ is a Hermitian matrix.

A square matrix $A$ is unitary if $A^HA=AA^H=I$, 
and is orthogonal if $A^TA=AA^T=I$.
For a vector ${\bf v}=(v_i)_{i=1}^n$ define the norms 
$||{\bf v}||_1=\sum_{i}|v_{i}|$,
$||{\bf v}||_2=(\sum_{i}|v_{i}|^2)^{1/2}$, and
$||{\bf v}||_{\infty}=\max_{i}|v_{i}|$.
For a matrix $A=(a_{ij})_{i,j=1}^{n}$ 
define its $h$-norms
$||A||_h=\inf_{||{\bf v}||_h=1}||A{\bf v}||_h$ for 
$h=1,2,\infty$ and its  Frobenius norm 
$||A||_F=(\sum_{i,j=1}^n|a_{ij}|^2)^{1/2}$. 
We write $||{\bf v}||=||{\bf v}||_2$ and $||A||=||A||_2$. 
$||A||$ is called the spectral 
norm of a matrix $A$.
For all matrices $A$ it holds that
\begin{equation}\label{eqnorm1inf}
||A||_1= ||A^T||_{\infty}= ||A^H||_{\infty}=\max_j\sum{i,j}|a_{ij}|,
\end{equation}
\begin{equation}\label{eqnorm12}
\frac{1}{\sqrt n}||A||_h\le||A||\le \sqrt n ||A||_h~{\rm for}~h=1~{\rm and}~h=\infty,~
||A||^2\le||A||_1||A||_{\infty}, 
\end{equation}
\begin{equation}\label{eqdiag}
||D||=||D||_1=||D||_{\infty},~
||D||_F^2=\sum_{i=1}^n|d_i|^2~
{\rm for~
a~diagonal~matrix}~
D=\diag(d_i)_{i=1}^n, 
\end{equation}
\begin{equation}\label{eqfrobun}
||UAV||=||A||~{\rm and} ||UAV||_F=||A||_F~{\rm for~unitary~matrices}~U~{\rm and}~V, 
\end{equation}
\begin{equation}\label{eqfrob}
||A||\le||A||_F\le \sqrt {\rho} ||A||~{\rm where}~\rho=\rank (A),
\end{equation}
\begin{equation}\label{eqnorm12inf}
||A+B||_h\le ||A||_h+||B||_h,~||FG||_h\le ||F||_h||G||_h~{\rm for}~h=1,2,\infty.
\end{equation}
For the latter bounds the matrix sizes must match, 
to define the matrices $A+B$ and $FG$.


$\kappa (A)=||A||~||A^{-1}||$ is the {\em condition 
number} of a nonsingular matrix $A$. Such matrix is {\em ill conditioned} 
if its condition number is large and is {\em well conditioned}
if it is reasonably bounded. 


\section{Toeplitz
and $f$-circulant 
matrices and their inverses}\label{stplc}


${\bf e}_i$ is the $i$th coordinate vector of a dimension $n$ for
$i=1,\dots,n$. Define two vectors
\begin{equation}\label{eqt_+}
{\bf t_+}=(t_{i})_{i=1-n}^{n-1}~{\rm and}~{\bf t}=(t_i)_{i=0}^{n-1}
\end{equation}
of dimensions $n$ and $2n-1$, respectively. They in turn
define 
a {\em Toep\-litz} $n\times n$ matrix $T_n=T({\bf t_+})=(t_{i-j})_{i,j=1}^{n}$, 
 a lower {\em triangular Toep\-litz}  $n\times n$ matrix $Z({\bf t})=(t_{i-j})_{i,j=1}^n$
(where $t_k=0$ for $k<0$),
and its transpose $Z({\bf t})^T=(Z({\bf t}))^T$.
Define the $n\times n$ matrices
$Z=Z_0=Z({\bf e}_2)$ of  downshift
and $Z_f=Z+f{\bf e}_1^T{\bf e}_n$, for  $f\neq 0$,
of $f$-{\em circular shift} 
(see equation (\ref{eqtz}) below).
 It holds that 
$Z{\bf t}=(t_{i-1})_{i=0}^{n-1}$ and
$Z({\bf t})=Z_0({\bf t})=\sum_{i=1}^{n}t_{h}Z^{i-1}$.
An $f$-{\em circulant} (or factor circulant) 
matrix $Z_f({\bf t})=\sum_{i=1}^{n}t_iZ_f^{i-1}$ 
is a special Toep\-litz $n\times n$ matrix defined by its first column vector 
${\bf t}=(t_i)_{i=1}^{n}$ and a scalar $f$. 
$f$-circulant matrix is called {\em circulant} if $f=1$.
\begin{equation}\label{eqtz} 
T_n=\begin{pmatrix}t_0&t_{-1}&\cdots&t_{1-n}\\ t_1&t_0&\smallddots&\vdots\\ \vdots&\smallddots&\smallddots&t_{-1}\\ t_{n-1}&\cdots&t_1&t_0\end{pmatrix},~Z=\begin{pmatrix}
        0   &       &   \dots    &   & 0\\
        1   & \ddots    &       &   & \\
        \vdots     & \ddots    & \ddots    &   & \vdots    \\
            &       & \ddots    & 0 &  \\
        0    &       &  \dots      & 1 & 0 
    \end{pmatrix}, ~Z_f=\begin{pmatrix}
        0   &       &   \dots    &   & f\\
        1   & \ddots    &       &   & \\
        \vdots     & \ddots    & \ddots    &   & \vdots    \\
            &       & \ddots    & 0 &  \\
        0    &       &  \dots      & 1 & 0 
    \end{pmatrix}.
\end{equation}
$J=J_n=({\bf e}_n~|~\dots~|~{\bf e}_1)$ is the 
 reflection matrix of size $n\times n$, $J=J^T=J^{-1}$.
A Toeplitz matrix $T_n$ and its inverse (if
defined) are {\em persymmetric}, that is, 
\begin{equation}\label{eqpers}
JT_nJ=T_n~{\rm and}~JT_n^{-1}J=T_n^{-1}.
\end{equation}


Hereafter $\omega=\exp(2\pi\sqrt{-1}/n)$ is a primitive $n$th root of unity.
$\Omega=(\omega^{(i-1)(j-1)})_{i,j=0}^{n-1}$ is the matrix of the discrete Fourier transform
at $n$ points. 
\begin{theorem}\label{thdft}
$\Omega^H\Omega=nI$, that is
 $\frac{1}{\sqrt n}\Omega$ 
is a unitary matrix.
\end{theorem}
The following theorem implies that the inverses 
 (wherever they are defined) and pairwise  products of  
$f$-circulant matrices are $f$-circulant.

\begin{theorem}\label{thcpw} (See \cite{CPW74}.)
We have 
$Z_1({\bf t})=\Omega^{-1}D(\Omega{\bf t})\Omega.$
More generally, for any $f\ne 0$, we have
$Z_{f^n}({\bf t})=U_f^{-1}D(U_f{\bf t})U_f$
where
$U_f=\Omega D({\bf f}),~~{\bf f}=(f^i)_{i=0}^{n-1}$,
$D({\bf u})=\diag(u_i)_{i=0}^{n-1}$ for a vector ${\bf u}=(u_i)_{i=0}^{n-1}$.  
\end{theorem}


\begin{theorem}\label{thgs}
Write $T_{k}=(t_{i-j})_{i,j=0}^{k-1}$ for  $k=n,n+1$. 

(a) Let the matrix $T_n$ be nonsingular and write 
${\bf p}=T_n^{-1}{\bf e}_1$ and ${\bf q}=T_n^{-1}{\bf e}_{n}$.
If
$p_{1}={\bf e}_1^T{\bf p}\neq 0$,
then
$p_{1}T_n^{-1}=Z({\bf p})Z(J{\bf q})^T-Z(Z{\bf q})Z(ZJ{\bf p})^T.$

In parts (b) and (c) below let a Toeplitz $(n+1)\times (n+1)$ 
matrix $T_{n+1}$ be nonsingular and write 
$\widehat {\bf v}=(v_i)_{i=0}^n=T_{n+1}^{-1}{\bf e}_1$,
${\bf v}=(v_i)_{i=0}^{n-1}$, ${\bf v}'=(v_i)_{i=1}^{n}$,
$\widehat {\bf w}=(w_i)_{i=0}^n=T_{n+1}^{-1}{\bf e}_{n+1}$, 
${\bf w}=(w_i)_{i=0}^{n-1}$, and ${\bf w}'=(w_i)_{i=1}^{n}$.

(b) If $v_0\neq 0$, then the matrix $T_n$ is nonsingular and
$v_0T_n^{-1}=Z({\bf v})Z(J{\bf w'})^T-Z({\bf w})Z(J{\bf v}')^T$.

(c) If $v_n\neq 0$, then the matrix $T_{1,0}=(t_{i-j})_{i=1,j=0}^{n,n-1}$ is nonsingular and
$v_nT_{1,0}^{-1}=Z({\bf w})Z(J{\bf v'})^T-Z({\bf v})Z(J{\bf w}')^T$.
\end{theorem}
\begin{proof}
 See \cite{GS72} on parts (a) and (b);  see \cite{GK72} on part (c).
\end{proof}


\section{The norm bounds for  circulant and Toeplitz matrices}\label{srnbtc}


\begin{theorem}\label{thctn}
Assume a pair of vectors ${\bf t}$ and 
${\bf t_+}$ of (\ref{eqt_+}), defining the
 circulant and Toeplitz $n\times n$ matrices
 $Z_1({\bf t})$ and  $T_n=T({\bf t_+})$. 
Then it holds that
\begin{equation}\label{eqttn}
||Z({\bf t})||_h\le ||Z_1({\bf t})||_h~{\rm for}~h=F,1,2,\infty,
\end{equation}
\begin{equation}\label{eqcrlnb}
||Z_1({\bf t})||\le ||Z_1({\bf t})||_1=||Z_1({\bf t})||_{\infty}=||{\bf t}||_1\le \sqrt n~||{\bf t}||,
\end{equation}
\begin{equation}\label{eqcrcf2}
 ||Z_1({\bf t})||_F=\sqrt n~||{\bf t}||,
\end{equation}
\begin{equation}\label{eqtn12}
 ||T_n||\le  ||T_n||_1=||T_n||_{\infty}\le||{\bf t_+}||_1\le \sqrt {2n-1}~ ||{\bf t_+}||~{\rm and}~
 ||T_n||_F\le\sqrt {2n-1}~||{\bf t_+}||.
\end{equation}
\end{theorem}
\begin{proof}
Readily verify equation (\ref{eqttn}).
Combine equations 
(\ref{eqnorm1inf}) and  (\ref{eqnorm12}) and 
obtain relationships 
(\ref{eqcrlnb}).
Combine Theorems \ref{thdft} and \ref{thcpw} with relationships 
(\ref{eqdiag}) and (\ref{eqfrobun}) and obtain that
$ ||Z_1({\bf t})||_F^2=
||D(\Omega{\bf t})||_F^2=||(\Omega{\bf t})||^2=n~||{\bf t}||^2$,  
yielding  
equation (\ref{eqcrcf2}).
Embed the  $n\times n$ matrix $T_n$
into the circulant $(2n-1)\times (2n-1)$ matrix $Z_1({\bf t_+})$
and obtain that
$||T_n||_h\le ||Z_1({\bf t_+})||_h$ where $h$ can stand for $F$,1,2, and $\infty$.
Together with  relationships (\ref{eqcrlnb}) and (\ref{eqcrcf2}), this implies relationships 
(\ref{eqtn12}).
\end{proof}




\begin{remark}\label{refcrc} (Extension to the case of $f$-circulant  matrices.)
Theorem  \ref{thcpw} implies that 
$$\frac{1}{g(f)}||Z_1({\bf v})||\le ||Z_f({\bf v})||\le g(f) ||Z_1({\bf v})||$$
and if the matrices $Z_1({\bf v})$ and $Z_f({\bf v})$ are nonsingular, then also
$$\frac{1}{g(f)}||Z_1({\bf v})^{-1}||\le ||Z_f({\bf v})^{-1}||\le g(f) ||Z_1({\bf v})^{-1}||$$
for all vectors ${\bf v}$, scalars $f\neq 0$,
$g(f)=\max\{|f|,{1/|f|}\}$, and $j=1,\dots,n$. Therefore we can readily extend
our  norm estimates from circulant to $f$-circulant matrices for $f\neq 0$.
In particular the former estimates do not change 
 where $|f|=1$.
\end{remark}


\begin{remark}\label{refh} (Extension to the case of Hankel  matrices.)
All our estimates for Toeplitz matrices are immediately extended to the case of Hankel 
matrices
$H_n=(h_{i+j})_{i,j=1}^{n}$ because 
 the products $H_nJ=T_n$ and $JH_n=T_n$ are
$n\times n$ Toep\-litz 
matrices.
\end{remark}


\section{The norm bounds for the inverses of Toeplitz and circulant matrices}\label{sircrc}


Hereafter we write $((v_i)_{i=1}^n)^{-1}=(1/v_i)_{i=1}^n)$.
Similarly to equation (\ref{eqcrcf2})
deduce that

\begin{equation}\label{eqicf2}
||Z_1({\bf t})^{-1}||_F=
||(D(\Omega{\bf t}))^{-1}||_F=||(\Omega{\bf t})^{-1}||,~
||Z_1({\bf t})^{-1}||= ||(D(\Omega{\bf t}))^{-1}||=||{ (\Omega\bf t})^{-1}||_{\infty}.
\end{equation}
Combine this bound with (\ref{eqnorm12}) and obtain that

$$||Z_1({\bf t})^{-1}||_h\le \sqrt n~||(\Omega{\bf t})^{-1}||_{\infty}~{\rm for}~h=1~{\rm and}~h=\infty.$$

\begin{theorem}\label{thnti}
Under the assumptions of part (a) of Theorem \ref{thcpw} it
holds that
$||p_{1}T_n^{-1}||_h\le 2 ||{\bf p}||_1~||{\bf q}||_1\le 2n ||{\bf p}||~||{\bf q}||~{\rm for}~h=1,2,\infty$.
\end{theorem}


\begin{proof}  
Recall from part (a) of Theorem  \ref{thgs} that
$p_{1}T_n^{-1}=Z({\bf p})Z(J{\bf q})^T-Z(Z{\bf q})Z(ZJ{\bf p})^T$.
Therefore (cf. (\ref{eqnorm12inf}))
$||p_{1}T_n^{-1}||_h\le ||Z({\bf p})||_h~||Z(J{\bf q})^T||_h+||Z(Z{\bf q})||_h~||Z(ZJ{\bf p})^T||_h$
for $h=1,2,\infty$.
Combine this bound with relationships 
(\ref{eqttn}) and 
(\ref{eqcrlnb})
and deduce  that
$||p_{1}T_n^{-1}||_h\le||{\bf p}||_1~||J{\bf q}||_1+||Z{\bf q}||_1~||ZJ{\bf p}||_1$
due to (\ref{eqcrlnb}).
Note that $||J{\bf w}||_1=||{\bf w}||_1$ and $||Z{\bf w}||_1\le ||{\bf w}||_1$
for every vector ${\bf w}$ and obtain Theorem \ref{thnti}.
\end{proof} 
 

\section{Gaussian random matrices
}\label{grm}


By extending the norm bounds of the two previous sections 
we can estimate the norms of circulant and Toeplitz matrices
with random entries selected
under various probability distributions. We are going to do this just
under the Gaussian distribution, which enables simpler and stronger estimates.

\begin{definition}\label{defcdf}
$F_{\gamma}(y)=$ Probability$\{\gamma\le y\}$
is the {\em cumulative 
distribution function (cdf)} of a real random variable $\gamma$ evaluated at a real point $y$. 
$F_{g(\mu,\sigma)}(y)=\frac{1}{\sigma\sqrt {2\pi}}\int_{-\infty}^y \exp (-\frac{(x-\mu)^2}{2\sigma^2}) dx$ 
for a Gaussian random variable $g(\mu,\sigma)$ with a mean $\mu$, a positive variance $\sigma^2$,
 and a positive standard deviation $\sigma$.
\end{definition}
It follows that   
\begin{equation}\label{eqnormal}
|y-\mu|/\sigma\ge s~{\rm with ~the~probability}~\sqrt{\frac{2}{\pi}}\int_s^{+\infty} \exp (-\frac{x^2}{2}) dx,
\end{equation}
that is the probability 
decays very fast as $s$ grows large.   


\begin{definition}\label{defrndm}
A matrix (or a vector) is a {\em Gaussian random matrix (or vector)} with a mean 
$\mu$ and a positive variance $\sigma^2$ if it is filled with 
independent identically distributed Gaussian random 
variables, all having the mean $\mu$ and variance $\sigma^2$. 
$\mathcal G_{\mu,\sigma}^{m\times n}$ is the set of such
Gaussian  random  $m\times n$ matrices,
which are {\em standard} where $\mu=0$
and $\sigma^2=1$. By restricting this set 
to Toeplitz or $f$-circulant matrices we obtain the sets
$\mathcal T_{\mu,\sigma}^{m\times n}$ and
$\mathcal Z_{f,\mu,\sigma}^{n\times n}$ of
{\em Gaussian random Toep\-litz} 
and {\em Gaussian random $f$-circulant matrices}, 
respectively, which are standard where $\mu=0$
and $\sigma^2=1$.  
\end{definition}


\begin{definition}\label{defchi}
$\chi_{\mu,\sigma,n}(y)$ is the cdf of the norm 
$||{\bf v}||=(\sum_{i=1}^n v_i^2)^{1/2}$ of a Gaussian random vector
${\bf v}=(v_i)_{i=1}^n\in \mathcal G_{\mu,\sigma}^{n\times 1}$. For 
 $y\ge 0$ we have 
$\chi_{0,1,n}(y)= \frac {2}{2^{n/2}\Gamma(n/2)}\int_{0}^yx^{n-1}\exp(-x^2/2) dx$ 
where $\Gamma(h)=\int_0^{\infty}x^{h-1}\exp(-x) dx$ is the Gamma function, 
$\Gamma (n+1)=n!$ for nonnegative integers $n$.
\end{definition}

We recall the following basic results.


\begin{lemma}\label{lepr3}
Suppose $G$ is Gaussian matrix, 
$S$ and $T$ are square orthogonal matrices,
and the products $SG$ and $GT$ are well defined.
Then $SG$ and $GT$ are Gaussian matrices.
\end{lemma}


\begin{lemma}\label{leinp} \cite[Lemma A.2]{SST06}.
For a nonnegative scalar $y$, a unit vector ${\bf t}\in \mathbb R^{n\times 1}$, and a vector
 ${\bf b}\in \mathcal G_{\mu,\sigma}^{n\times 1}$, 
 we have  
$F_{|{\bf t}^T{\bf b}|}(y)\le \sqrt{\frac{2}{\pi}}\frac{y}{\sigma}$.
\end{lemma}  


\begin{remark}\label{reinp}
The latter bound is independent of $\mu$ and $n$.
It holds for any $\mu$ even if 
all coordinates of the vector ${\bf b}$ are fixed except for a
single coordinate in $\mathcal G_{\mu,\sigma}$.
\end{remark}


Hereafter
we assume by default that
{\em Gaussian random 
 general, Toeplitz and circulant 
matrices  have full rank}
(see Appendix \ref{srannns}).


\section{Norm bounds for Gaussian random Toeplitz and circulant matrices
and their inverses}\label{sgrtc}


\subsection{The norms of Gaussian random Toeplitz and circulant matrices}\label{scgrtm}


Combine our estimates of Theorem \ref{thctn} with Definition \ref{defrndm} and obtain the following 
upper bounds on the norms of
Gaussian random Toeplitz and circulant matrices.
\begin{corollary}\label{corgtcn}
For Gaussian random Toeplitz $n\times n$ matrix $T_n$
and Gaussian random circulant $n\times n$ matrix $Z_1({\bf t})$
it holds that
$F_{||T_n||_h}(y)\ge \chi_{\mu,\sigma,2n-1}(y/\sqrt {2n-1})$
and $F_{||Z_1({\bf t}||_h}(y)\ge \chi_{\mu,\sigma,n}(y/\sqrt {n})$
where $h$ can stand for $F$,1,2, or $\infty$.
\end{corollary}


\subsection{The expected norms of the inverses of Gaussian random matrices
have order at least $1/\sigma$ where the mean $\mu$ greatly exceeds 
the standard deviation $\sigma$}\label{scgrtmill}


Bounds (\ref{eqnormal}) 
 imply that the matrix 
$M\in \mathcal G_{\mu,\sigma}^{n\times n}$ is expected to be approximated within the norm
bound of order  $\sigma$
by 
the  matrix $\mu {\bf e}{\bf e}^T$ of rank 1, for the vector 
${\bf e}=(1,\dots,1)^T$ provided that $\\mu|\gg \sigma$,
and if indeed so, then 
the norm $||M^{-1}||$
has order of at least $1/\sigma$. Similar argument applies to 
Gaussian random Toeplitz and circulant matrices.
In the nexct section we derive our estimates for any pair of $(\mu, \sigma)$, but one 
 can avoid the latter undesired growth of the norm 
 by restricting the study
 to standard 
Gaussian random Toeplitz and circulant matrices.


\subsection{The norm of the inverse of Gaussian random circulant matrix}\label{scgrcm}




\begin{theorem}\label{thcircsing} (Cf. Table \ref{tabcondcirc}.)
Assume $y\ge 0$ and vector 
 ${\bf t}\in \mathcal G_{\mu,\sigma}^{n\times 1}$. Then
$F_{||Z_1({\bf t})^{-1}||}(z)=1-(1-q)^n$, for
$q=\frac{1}{\sigma\sqrt {2\pi}}\int_{-1/z}^{1/z} \exp (-\frac{(x-\mu)^2}{2\sigma^2}) dx$,
and so $F_{||Z_1({\bf t})^{-1}||}(z)\approx nq$ where the value $q$ is small, that is where
the value $|z|$
is large or the distance between $\mu$ and the range $[-1/z,1/z]$ is large.
\end{theorem} 


\begin{remark}\label{ren1etc}
By applying bounds (\ref{eqnorm12}) and (\ref{eqfrob})
we can extend the theorem to estimate 
the cdfs $F_{||Z_1({\bf t})^{-1}||_h}(z)$ where $h$ can stand for $F$,1, or $\infty$.
\end{remark}


\begin{proof}
Equation (\ref{eqicf2}) implies that 
$||Z_1({\bf t})^{-1}||=1/\min_{i=1}^{n} |u_i|$ where ${\bf u}=(u_i)_{i=1}^{n}=\Omega {\bf t}$. 
Apply Lemma \ref{lepr3} and obtain that
${\bf u}$ is a Gaussian random vector ${\bf q}(\mu,\sigma)$.
For any $i$, $i=1,\dots,n$, it holds that $|u_i|\le y$ with probability 
$1-p$ where $p=\frac{1}{\sigma\sqrt {2\pi}}\int_{-y}^y \exp (-\frac{(x-\mu)^2}{2\sigma^2}) dx$
(cf. Definition \ref{defcdf}),
and so $\min_{i=1}^{n}|u_i|\le y$ with probability 
$(1-p)^n$ because $u_i$ are independent random variables.
Equivalently  $1/\min_{i=1}^{n}|u_i|\ge 1/y$ with probability 
$(1-p)^n$.
Therefore 
$||Z_1({\bf t})^{-1}||=1/\min_{i=1}^{n} |u_i|\le 1/y$
with probability 
$1-(1-p)^n\ge np$. Substitute $z=1/y$
and obtain the theorem.
\end{proof}





\subsection{Norm bounds for the inverse of Gaussian random Toeplitz matrix}\label{sigrtm}


Our next subject is  
the estimates for 
 the norm $||T_n^{-1}||_h$ 
for $h=1,2,\infty$ and Gaussian random Toeplitz matrix
$T_{n}\in \mathcal T_{\mu,\sigma}^{n\times n}$,
which is known to be
nonsingular with probability 1. We can extend these estimates to 
 the norm $||T_n^{-1}||$ by using (\ref{eqfrob}).


\begin{theorem}\label{thsigunat1}  
Given a matrix 
$T_{n}=(t_{i-j})_{i,j=1}^n\in \mathcal T_{0,1}^{n\times n}$,
assumed to be nonsingular,
write ${\bf p}=(p_i)_{i=1}^n=T_n^{-1}{\bf e}_1$,  ${\bf q}=(q_i)_{i=1}^n=T_n^{-1}{\bf e}_1$,
$u_n=p_n/||{\bf p}||$, $v_1=q_1/||{\bf q}||$, and
$p_{1}=
{\bf e}_n^T{\bf p}={\bf e}_1^T{\bf q}=q_n$ 
(cf. (\ref{eqpers})).
Then $||p_{1}T_n^{-1}||_h\le 2n/(\alpha \beta)$ 
 for $h=1,2,\infty$ and two random variables  $\alpha$ and $\beta$
such that 
\begin{equation}\label{eqprtinv}
F_{\alpha}(y)\le \sqrt{\frac{2n}{\pi}}\frac{y}{\sigma|u_n|}~{\rm and}~
F_{\beta}(y)\le \sqrt{\frac{2n}{\pi}}\frac{y}{\sigma|v_1|}~{\rm for}~y\ge 0.
\end{equation}    
\end{theorem}


\begin{proof}  
By virtue of Theorem \ref{thnti} we just need to estimate the
two random variables
 $||{\bf p}||$ and $||{\bf q}||$.  
By virtue of its definition the vector ${\bf p}$  is orthogonal
to the vectors $T_n{\bf e}_2,\dots,T_n{\bf e}_n$, 
whereas ${\bf p}^TT_n{\bf e}_1=1$ (cf. \cite{SST06}).
Consequently the vectors $T_n{\bf e}_2,\dots,T_n{\bf e}_n$
uniquely define the vector 
 ${\bf u}=(u_i)_{i=1}^n={\bf p}/||{\bf p}||$,
whereas 
$|{\bf u}^TT_n{\bf e}_1|=1/||{\bf p}||$.
The last coordinate $t_{n-1}$ of the vector $T_n{\bf e}_1$
is independent of the vectors $T_n{\bf e}_2,\dots,T_n{\bf e}_n$
and consequently of the vector ${\bf u}$. 
Apply 
 Remark \ref{reinp} to estimate the cdf of the random 
variable $\alpha/|u_n|=1/(||{\bf p}||~|u_n|)=|{\bf u}^TT_n{\bf e}_1|/|u_n|$  
 and obtain that
$F_{\alpha}(y)\le  \sqrt{\frac{2n}{\pi}}\frac{y}{\sigma|u_n|}$ for $y\ge 0$.
Likewise
we deduce that
$F_{\beta}(y)\le  \sqrt{\frac{2n}{\pi}}\frac{y}{\sigma|v_1|}$ for $y\ge 0$.
Finally combine both bounds (\ref{eqprtinv}) on the cdfs $F_{\alpha/|u_n|}(y)$ and 
$F_{\beta/|v_1|}(y)$ with Theorem \ref{thnti}.
\end{proof}

\subsection{Bounding the leading entry of the inverse}\label{sgcppr}
  

Theorem \ref{thsigunat1} bounds   
 the norm  $||T_n^{-1}||$,
 in terms of the random variables 
$|u_n|$, $|v_1|$, and $|p_1|=|q_n|$. 
By applying parts (b) and (c)
of Theorem  \ref{thgs} instead of its part (a),
we similarly
deduce the
bounds $||v_0T_{n+1}^{-1}||\le 2/(\alpha\beta)$ and
 $||v_nT_{n+1}^{-1}||\le 2/(\alpha\beta)$
for two pairs of random variables $\alpha$ and $\beta$
that
satisfy (\ref{eqprtinv}) for $n+1$ replacing $n$.
Note that  $p_{1}=\frac{\det T_{n-1}}{\det T_{n}}$,
$v_0=\frac{\det T_n}{\det T_{n+1}}$, and
$v_n=\frac{\det T_{0,1}}{\det T_{n+1}}$
for $T_{0,1}=(t_{i-j})_{i=0,j=1}^{n-1,n}$.
Next we bound the 
geometric means 
of the 
ratios 
$|\frac{\det T_{h+1}}{\det T_{h}}|$
for
$h=1,\dots,k-1$. 
 $1/|p_1|$ and  $1/|v_0|$
are such  ratios for $k=n-1$ and $k=n$,
respectively,
whereas the  ratio  $1/|v_n|$ is similar to 
$1/|v_0|$, under slightly distinct notation. 



\begin{theorem}\label{thhdmr} 
 Let $T_h\neq O$ denote $h\times h$ matrices
for $h=1,\dots,k$  
whose entries have absolute values at most $t$
for a random variable $t$, e.g., for $t=||T||$.
Furthermore let $T_1=(t)$.
Then the geometric mean $(\prod_{h=1}^{k-1}|\frac{\det T_{h+1}}{\det T_{h}}|)^{1/(k-1)}=\frac{1}{t}|\det T_{k}|^{1/(k-1)}$
is at most $k^{\frac{1}{2}(1+\frac{1}{k-1})}t$.
\end{theorem}


\begin{proof}
The theorem follows from 
Hadamard's upper bound
$|\det M|\le k^{k/2}t^k$, which holds
for any $k\times k$ matrix $M=(m_{i,j})_{i,j=1}^k$
with $\max_{i,j=1}^k|m_{i,j}|\le t$.
\end{proof}
 
The theorem shows that
the geometric mean of the ratios $|\frac{\det T_{h+1}}{\det T_{h}}|$
for 
$h=1,\dots,k-1$
 is not greater than $k^{0.5+\epsilon(k)}t$
where $\epsilon(k)\rightarrow 0$ as $k\rightarrow \infty$.
Furthermore if
$T_n\in \mathcal T_{\mu,\sigma}^{n\times n}$
we can write
$t$  
 and 
 recall Definition \ref{defcdf} to define the cdf of the Gaussian random variable $t$. 
This implies a reasonable lower bound on the expected value $|p_1|=|q_n|$.

\subsection{The generic corner property, empirical results, and a link of Toeplitz inversion 
to polynomial computations}\label{srvrm}
  
Our study in the two previous subsections implies that
the norms of the inverses of
Gaussian random  Toeplitz matrices 
and consequently their condition numbers 
are expected to be reasonably bounded
provided that the values $|u_n|$ and $|v_1|$ are not small.
We call the latter provision the {\em generic property of two corners
of the inverse}. 
Proving it may be hard. No proof
could work if the mean value $\mu$
greatly exceeds
the standard deviation $\sigma$
because in this case
 the norm $||T_n^{-1}||$ is likely to be 
large 
(see Section \ref{scgrtmill}), which would contradict to
the generic property by virtue of Theorems \ref{thsigunat1} and \ref{thhdmr}.
Empirically, however,
standard Gaussian random  Toeplitz matrices 
of reasonable sizes do
tend to be reasonably well conditioned.
Our tests
in the next section show this 
directly,
whereas the numerical
tests in \cite{XXG12} and
 \cite{PQZ13},
provide additional indirect support.
Namely these tests succeed by employing
random  Toeplitz multipliers,
whereas they would have been expected to fail numerically 
if the multipliers were ill conditioned.

We conclude this section by
 recalling a link of Toeplitz
matrix inversion to 
some polynomial computations.
Namely it is well known (see \cite[Section 2.11]{P01})
that the equation $T_n{\bf p}={\bf e}_1$ is equivalent to 
the polynomial equation
$t(x)p(x)\mod x^{2n+1}=r(x)$
where $r(x)$ is a monic polynomial of degree $n$,
whereas
$p(x)$ and $t(x)$ are two polynomials of degrees $n$
and $2n-1$, respectively, with the coefficient vectors 
${\bf p}$ and ${\bf t_+}$, respectively.
For a given   vector ${\bf t_+}$, 
the Euclidean algorithm computes
the coefficients of the polynomials $p(x)$ and $r(x)$ 
(apart from the cases of degeneracy, occurring
with probability 0 for Gaussian random input) but provides no explicit expressions for
these coefficients.



\section{Numerical Experiments}\label{sexp}

 
Our numerical experiments with random general, Hankel, Toeplitz and circulant matrices 
have been performed in the Graduate Center of the City University of New York 
on a Dell server with a dual core 1.86 GHz
Xeon processor and 2G memory running Windows Server 2003 R2. The test
Fortran code was compiled with the GNU gfortran compiler within the Cygwin
environment.  Random numbers were 
 generated 
with the 
random
\_
number
intrinsic Fortran function, 
assuming the 
uniform probability distribution 
over the range $\{x:-1 \leq x < 1\}$.  
The tests have been designed by the first author 
and performed by his coauthor.


We have computed the condition numbers of Gaussian random general $n\times n$ matrices for 
$n=2^k$, $k=5,6,\dots,$ 
with the entries sampled in the range $[-1,1)$ 
as well as 
complex general, Toeplitz, and circulant matrices 
whose entries had real and imaginary parts sampled at random in the same range $[-1,1)$. 
We performed 100 tests for each class of inputs, each dimension $n$,  and each nullity $r$.
 Tables \ref{tab01}--\ref{tabcondcirc} display
 the test results. The last four columns of each table 
display the average (mean), minimum, maximum, and standard deviation
of the computed condition numbers of the input matrices, respectively. Namely we 
computed
the values $\kappa (A)=||A||~||A^{-1}||$ for general, Toeplitz, and circulant matrices $A$ and
the values $\kappa_1 (A)=||A||_1~||A^{-1}||_1$ for Toeplitz matrices $A$.
We computed and displayed in Table \ref{tabcondtoep} the condition numbers 
$\kappa_1(A)=||A||_1||A^{-1}||_1$ by using 1-norms of 
Toeplitz matrices and their inverses
rather than their spectral norms.
This shift from 2-norms to 1-norms  facilitated the computations in the case of the inputs of large sizes
and made no significant impact on the output condition numbers.
Indeed already relationships (\ref{eqnorm12}) link 
 the 1-norms and 2-norms to one another, but 
the empirical data in 
Table \ref{nonsymtoeplitz} consistently show 
even much closer links,
in all cases of
general, Toeplitz, and circulant  $n\times n$ matrices $A$ where
$n=32,64,\dots, 1024$.   
Our estimates in Section \ref{scgrcm}
are sharp, and so we were most interested in testing the condition numbers of 
random Toeplitz matrices. 
The displayed data show that these condition numbers grow 
substantially faster with the growth of the
matrix  size than in the case of random circulant matrices but not faster than in the case of
random general
matrices, expected to be well conditioned according 
to the results of the intensive formal and experimental study in 
\cite{D88}, \cite{E88}, \cite{ES05}, \cite{CD05},  \cite{SST06}.


\begin{table}[h]
\caption{The norms of random general, Toeplitz and circulant $n\times n$ matrices and of their inverses}
\label{nonsymtoeplitz}
 \begin{center}
\begin{tabular}{|c|c|c|c|c|c|c|c|}
\hline
\textbf{matrix $A$}&\textbf{$n$}&\textbf{$||A||_1$}&\textbf{$||A||_2$}&\textbf{$\frac{||A||_1}{||A||_2}$}&\textbf{$||A^{-1}||_1$}&\textbf{$||A^{-1}||_2$}&\textbf{$\frac{||A^{-1}||_1}{||A^{-1}||_2}$}\\\hline
General & $32$ & $1.9\times 10^{1}$ & $1.8\times 10^{1}$ & $1.0\times 10^{0}$ & $4.0\times 10^{2}$ & $2.1\times 10^{2}$ & $1.9\times 10^{0}$ \\ \hline
General & $64$ & $3.7\times 10^{1}$ & $3.7\times 10^{1}$ & $1.0\times 10^{0}$ & $1.2\times 10^{2}$ & $6.2\times 10^{1}$ & $2.0\times 10^{0}$ \\ \hline
General & $128$ & $7.2\times 10^{1}$ & $7.4\times 10^{1}$ & $9.8\times 10^{-1}$ & $3.7\times 10^{2}$ & $1.8\times 10^{2}$ & $2.1\times 10^{0}$ \\ \hline
General & $256$ & $1.4\times 10^{2}$ & $1.5\times 10^{2}$ & $9.5\times 10^{-1}$ & $5.4\times 10^{2}$ & $2.5\times 10^{2}$ & $2.2\times 10^{0}$ \\ \hline
General & $512$ & $2.8\times 10^{2}$ & $3.0\times 10^{2}$ & $9.3\times 10^{-1}$ & $1.0\times 10^{3}$ & $4.1\times 10^{2}$ & $2.5\times 10^{0}$ \\ \hline
General & $1024$ & $5.4\times 10^{2}$ & $5.9\times 10^{2}$ & $9.2\times 10^{-1}$ & $1.1\times 10^{3}$ & $4.0\times 10^{2}$ & $2.7\times 10^{0}$ \\ \hline
Toeplitz & $32$ & $1.8\times 10^{1}$ & $1.9\times 10^{1}$ & $9.5\times 10^{-1}$ & $2.2\times 10^{1}$ & $1.3\times 10^{1}$ & $1.7\times 10^{0}$ \\ \hline
Toeplitz & $64$ & $3.4\times 10^{1}$ & $3.7\times 10^{1}$ & $9.3\times 10^{-1}$ & $4.6\times 10^{1}$ & $2.4\times 10^{1}$ & $2.0\times 10^{0}$ \\ \hline
Toeplitz & $128$ & $6.8\times 10^{1}$ & $7.4\times 10^{1}$ & $9.1\times 10^{-1}$ & $1.0\times 10^{2}$ & $4.6\times 10^{1}$ & $2.2\times 10^{0}$ \\ \hline
Toeplitz & $256$ & $1.3\times 10^{2}$ & $1.5\times 10^{2}$ & $9.0\times 10^{-1}$ & $5.7\times 10^{2}$ & $2.5\times 10^{2}$ & $2.3\times 10^{0}$ \\ \hline
Toeplitz & $512$ & $2.6\times 10^{2}$ & $3.0\times 10^{2}$ & $8.9\times 10^{-1}$ & $6.9\times 10^{2}$ & $2.6\times 10^{2}$ & $2.6\times 10^{0}$ \\ \hline
Toeplitz & $1024$ & $5.2\times 10^{2}$ & $5.9\times 10^{2}$ & $8.8\times 10^{-1}$ & $3.4\times 10^{2}$ & $1.4\times 10^{2}$ & $2.4\times 10^{0}$ \\ \hline   
Circulant & $32$ & $1.6\times 10^{1}$ & $1.8\times 10^{1}$ & $8.7\times 10^{-1}$ & $9.3\times 10^{0}$ & $1.0\times 10^{1}$ & $9.2\times 10^{-1}$ \\ \hline
Circulant & $64$ & $3.2\times 10^{1}$ & $3.7\times 10^{1}$ & $8.7\times 10^{-1}$ & $5.8\times 10^{0}$ & $6.8\times 10^{0}$ & $8.6\times 10^{-1}$ \\ \hline
Circulant & $128$ & $6.4\times 10^{1}$ & $7.4\times 10^{1}$ & $8.6\times 10^{-1}$ & $4.9\times 10^{0}$ & $5.7\times 10^{0}$ & $8.5\times 10^{-1}$ \\ \hline
Circulant & $256$ & $1.3\times 10^{2}$ & $1.5\times 10^{2}$ & $8.7\times 10^{-1}$ & $4.7\times 10^{0}$ & $5.6\times 10^{0}$ & $8.4\times 10^{-1}$ \\ \hline
Circulant & $512$ & $2.6\times 10^{2}$ & $3.0\times 10^{2}$ & $8.7\times 10^{-1}$ & $4.5\times 10^{0}$ & $5.4\times 10^{0}$ & $8.3\times 10^{-1}$ \\ \hline
Circulant & $1024$ & $5.1\times 10^{2}$ & $5.9\times 10^{2}$ & $8.7\times 10^{-1}$ & $5.5\times 10^{0}$ & $6.6\times 10^{0}$ & $8.3\times 10^{-1}$ \\ \hline
\end{tabular}
\end{center}
\end{table}


\begin{table}[h]
  \caption{The condition numbers $\kappa (A)$ of random $n\times n$ matrices $A$}
  \label{tab01}
  \begin{center}
    \begin{tabular}{| c | c | c | c | c | c |}
    \hline
\bf{$n$}&\bf{input} & \bf{min}  &\bf{max} &\bf{mean} &\bf{std} \\ \hline

$ 32   $ & $ {\rm real} $ & $2.4\times 10^{1}$ & $1.8\times 10^{3}$ & $2.4\times 10^{2}$ & $3.3\times 10^{2}$ \\ \hline
$ 64   $ & $ {\rm real} $ & $4.6\times 10^{1}$ & $1.1\times 10^{4}$ & $5.0\times 10^{2}$ & $1.1\times 10^{3}$ \\ \hline
$ 128  $ & $ {\rm real} $ & $1.0\times 10^{2}$ & $2.7\times 10^{4}$ & $1.1\times 10^{3}$ & $3.0\times 10^{3}$ \\ \hline
$ 256  $ & $ {\rm real} $ & $2.4\times 10^{2}$ & $8.4\times 10^{4}$ & $3.7\times 10^{3}$ & $9.7\times 10^{3}$ \\ \hline
$ 512  $ & $ {\rm real} $ & $3.9\times 10^{2}$ & $7.4\times 10^{5}$ & $1.8\times 10^{4}$ & $8.5\times 10^{4}$ \\ \hline
$ 1024 $ & $ {\rm real} $ & $8.8\times 10^{2}$ & $2.3\times 10^{5}$ & $8.8\times 10^{3}$ & $2.4\times 10^{4}$ \\ \hline
$ 2048 $ & $ {\rm real} $ & $2.1\times 10^{3}$ & $2.0\times 10^{5}$ & $1.8\times 10^{4}$ & $3.2\times 10^{4}$ \\ \hline
    \end{tabular}
  \end{center}
\end{table}

\begin{table}[h]
\caption{The condition numbers $\kappa_1 (A)=||A||_1~||A^{-1}||_1$ of random Toeplitz 
 $n\times n$ matrices $A$}
\label{tabcondtoep}
\begin{center}
\begin{tabular}{|c|c|c|c|c|}
\hline
\textbf{$n$}&\textbf{min}&\textbf{mean}&\textbf{max}&\textbf{std}\\\hline
$256$ & $9.1\times 10^{2}$ & $9.2\times 10^{3}$ & $1.3\times 10^{5}$ & $1.8\times 10^{4}$  \\ \hline
$512$ & $2.3\times 10^{3}$ & $3.0\times 10^{4}$ & $2.4\times 10^{5}$ & $4.9\times 10^{4}$  \\ \hline
$1024$ & $5.6\times 10^{3}$ & $7.0\times 10^{4}$ & $1.8\times 10^{6}$ & $2.0\times 10^{5}$ \\ \hline
$2048$ & $1.7\times 10^{4}$ & $1.8\times 10^{5}$ & $4.2\times 10^{6}$ & $5.4\times 10^{5}$ \\ \hline
$4096$ & $4.3\times 10^{4}$ & $2.7\times 10^{5}$ & $1.9\times 10^{6}$ & $3.4\times 10^{5}$ \\ \hline
$8192$ & $8.8\times 10^{4}$ & $1.2\times 10^{6}$ & $1.3\times 10^{7}$ & $2.2\times 10^{6}$ \\ \hline
\end{tabular}
\end{center}
\end{table}

\begin{table}[h]
\caption{The condition numbers $\kappa (A)$ of random circulant  $n\times n$ matrices $A$}
\label{tabcondcirc}
\begin{center}
\begin{tabular}{|c|c|c|c|c|}
\hline
\textbf{$n$}&\textbf{min}&\textbf{mean}&\textbf{max}&\textbf{std}\\\hline
$256$ & $9.6\times 10^{0}$ & $1.1\times 10^{2}$ & $3.5\times 10^{3}$ & $4.0\times 10^{2}$ \\ \hline
$512$ & $1.4\times 10^{1}$ & $8.5\times 10^{1}$ & $1.1\times 10^{3}$ & $1.3\times 10^{2}$ \\ \hline
$1024$ & $1.9\times 10^{1}$ & $1.0\times 10^{2}$ & $5.9\times 10^{2}$ & $8.6\times 10^{1}$ \\ \hline
$2048$ & $4.2\times 10^{1}$ & $1.4\times 10^{2}$ & $5.7\times 10^{2}$ & $1.0\times 10^{2}$ \\ \hline
$4096$ & $6.0\times 10^{1}$ & $2.6\times 10^{2}$ & $3.5\times 10^{3}$ & $4.2\times 10^{2}$ \\ \hline
$8192$ & $9.5\times 10^{1}$ & $3.0\times 10^{2}$ & $1.5\times 10^{3}$ & $2.5\times 10^{2}$ \\ \hline
$16384$ & $1.2\times 10^{2}$ & $4.2\times 10^{2}$ & $3.6\times 10^{3}$ & $4.5\times 10^{2}$ \\ \hline
$32768$ & $2.3\times 10^{2}$ & $7.5\times 10^{2}$ & $5.6\times 10^{3}$ & $7.1\times 10^{2}$ \\ \hline
$65536$ & $2.4\times 10^{2}$ & $1.0\times 10^{3}$ & $1.2\times 10^{4}$ & $1.3\times 10^{3}$ \\ \hline
$131072$ & $3.9\times 10^{2}$ & $1.4\times 10^{3}$ & $5.5\times 10^{3}$ & $9.0\times 10^{2}$ \\ \hline
$262144$ & $6.3\times 10^{2}$ & $3.7\times 10^{3}$ & $1.1\times 10^{5}$ & $1.1\times 10^{4}$ \\ \hline
$524288$ & $8.0\times 10^{2}$ & $3.2\times 10^{3}$ & $3.1\times 10^{4}$ & $3.7\times 10^{3}$ \\ \hline
$1048576$ & $1.2\times 10^{3}$ & $4.8\times 10^{3}$ & $3.1\times 10^{4}$ & $5.1\times 10^{3}$ \\ \hline   
\end{tabular}
\end{center}
\end{table}


\clearpage

\appendix
{\bf {\LARGE {Appendix}}}


\section{Randomness and nonsingularity}\label{srannns}

The total degree of a multivariate monomial is the sum of its degrees
in all its variables. The total degree of a polynomial is the maximal total degree of 
its monomials.


\begin{lemma}\label{ledl} \cite{DL78}, \cite{S80}, \cite{Z79}.
For a set $\Delta$ of a cardinality $|\Delta|$ in any fixed ring  
let a polynomial in $m$ variables have a total degree $d$ and let it not vanish 
identically on this set. Then the polynomial vanishes in at most 
$d|\Delta|^{m-1}$ points. 
\end{lemma}


We assume that Gaussian random variables range 
over infinite sets $\Delta$,
usually over the real line or its interval. Then
the lemma implies that a nonzero polynomial vanishes with probability 0.
Consequently a square Gaussian random general, Toeplitz or circulant
matrix is nonsingular 
with probability 1
because its determinant is a polynomials
in the entries. 
Likewise rectangular
 Gaussian random general, Toeplitz and circulant 
matrices have full rank with probability 1,
and similarly under the other probability distributions
whose measures are absolutely continuous relatively to Lebesgue's
measure. These results can be also adapted to the case of 
probability distribution over finite sets  \cite{DL78}, \cite{S80}, \cite{Z79}.

$~$

{\bf Acknowledgements:}
Our research has been supported by NSF Grant CCF--1116736 and
PSC CUNY Awards 64512--0042 and 65792--0043.



\begin{thebibliography}{hspace{0.5in}}




\bibitem[BG05]{BG05} 
A. B\"ottcher, S. M. Grudsky,
{\em Spectral Properties of Banded Toeplitz  Matrices}, 
SIAM Publications, Philadelphia, 2005. 






\bibitem[CD05]{CD05}
Z. Chen, J. J. Dongarra, Condition Numbers of Gaussian Random Matrices,
{\em SIAM. J. on Matrix Analysis and Applications}, {\bf 27}, 603--620, 2005.


\bibitem[CPW74]{CPW74}
R. E. Cline, R. J. Plemmons, and G. Worm, 
Generalized Inverses of Certain {T}oeplitz Matrices,
{\em Linear Algebra and Its Applications,} {\bf 8}, 25--33, 1974.




\bibitem[D88]{D88}
J. Demmel,  
The Probability That a Numerical Analysis Problem Is Difficult,
{\em Math. of Computation}, {\bf 50}, 449--480, 1988.


\bibitem[DL78]{DL78}
R. A. Demillo, R. J. Lipton,
A Probabilistic Remark on Algebraic Program Testing,
{\em Information Processing Letters}, {\bf 7}, {\bf 4}, 193--195, 1978. 




\bibitem[E88]{E88}
A. Edelman, Eigenvalues and Condition Numbers of Random Matrices,
{\em SIAM J. on Matrix Analysis and Applications}, {\bf 9}, {\bf 4},
543--560, 1988.


\bibitem[ES05]{ES05}
A. Edelman, B. D. Sutton,  Tails of Condition Number Distributions,
{\em SIAM J. on Matrix Analysis and Applications}, {\bf 27}, {\bf 2},
547--560, 2005.


\bibitem[GK72]{GK72}
I. Gohberg, N. Y. Krupnick,
A Formula for the Inversion of Finite Toeplitz Matrices, 
{\em Matematicheskiie Issledovaniia} (in Russian), {\bf 7}, {\bf 2}, 272--283, 1972.


\bibitem[GL96]{GL96}
G. H. Golub, C. F. Van Loan,
{\em Matrix Computations},
Johns Hopkins University Press, Baltimore, Maryland, 1996 (third addition).


\bibitem[GS72]{GS72}
I. Gohberg, A. Sementsul,
On the Inversion of Finite Toeplitz Matrices and Their Continuous Analogs, 
{\em Matematicheskiie Issledovaniia} (in Russian), {\bf 7}, {\bf 2}, 187--224, 1972.




\bibitem[HMT11]{HMT11}
N. Halko, P. G. Martinsson, J. A. Tropp,
Finding Structure with Randomness: Probabilistic Algorithms
for Constructing Approximate Matrix Decompositions, 
{\em SIAM Review}, {\bf 53,~2}, 217--288, 2011.




\bibitem[P01]{P01}
V. Y. Pan,
{\em Structured Matrices and Polynomials: Unified Superfast Algorithms},
Birkh\"auser/Springer, Boston/New York, 2001.


\bibitem[PQZ13]{PQZ13}
V. Y. Pan, G. Qian, A. Zheng,
Randomized Preprocessing versus Pivoting, 
{\em Linear Algebra and Its Applications}, 
{\bf 438,~4}, 1883--1899, 2013.


\bibitem[S80]{S80}
J. T. Schwartz,
Fast Probabilistic Algorithms for Verification of Polynomial Identities, 
{\em Journal of ACM}, {\bf 27}, {\bf 4}, 701--717, 1980. 


\bibitem[S98]{S98}
G. W. Stewart,
{\em Matrix Algorithms, Vol I: Basic Decompositions},
SIAM, 
1998.


\bibitem[SST06]{SST06}
A. Sankar, D. Spielman, S.-H. Teng, 
Smoothed Analysis of the Condition Numbers and Growth Factors of Matrices, 
{\em SIAM J. on Matrix Analysis}, {\bf 28}, {\bf 2}, 446--476, 2006. 




\bibitem[XXG12]{XXG12}
J. Xia, Y. Xi, M. Gu, 
A Superfast Structured Solver for Toeplitz Linear Systems via Randomized Sampling,
{\em SIAM J. Matrix Anal. Appl.}, {\bf 33}, 837--858, 2012.


\bibitem[Z79]{Z79}
R. E. Zippel, Probabilistic Algorithms for Sparse Polynomials, 
{\em Proceedings of EUROSAM'79, Lecture Notes in Computer Science}, 
{\bf 72}, 216--226, Springer, Berlin, 1979.


\end{thebibliography}
\end{document}